\documentclass[a4paper,11pt]{amsart}

\usepackage[utf8]{inputenc}

\usepackage[english]{babel}
\usepackage{amsmath,amsfonts,amsthm,xcolor,amssymb}
\usepackage{tikz}
\usetikzlibrary{positioning}
\usepackage[footnotesize]{caption}
\usepackage{geometry}
\usepackage[mathscr]{euscript}
\usepackage[colorlinks=true, citecolor=blue]{hyperref}
\numberwithin{equation}{section}

\newcommand{\vect}[1]{\mathbf{#1}}
\newcommand{\bm}[1]{\boldsymbol{#1}}
\newcommand{\mean}{\operatornamewithlimits{\mathbb E}}
\newcommand{\V}{\mathscr{H}_{\vect n,\vect d}}
\newcommand{\Xnd}{\mathbb X_{\vect n, \vect d}}
\newcommand{\Vnd}{\mathbb V_{n, d}}

\theoremstyle{plain}
\newtheorem{theorem}{Theorem}[section]
\newtheorem{lemma}[theorem]{Lemma}
\newtheorem{corollary}[theorem]{Corollary}
\newtheorem{proposition}[theorem]{Proposition}
\theoremstyle{definition}
\newtheorem{remark}[theorem]{Remark}
\newtheorem{example}[theorem]{Example}

\title{Reach of Segre--Veronese Manifolds}

\author[P.~Breiding]{Paul Breiding}
\address[Breiding]{University Osnabr\"uck,
FB 06 Mathematik/Informatik/Physik
Albrechtstr.~28a,
49076 Osnabrück, Germany}
\email{pbreiding@uni-osnabrueck.de}

\author[S.~Eggleston]{Sarah Eggleston}
\address[Eggleston]{University Osnabr\"uck,
FB 06 Mathematik/Informatik/Physik
Albrechtstr.~28a,
49076 Osnabrück, Germany}
\email{seggleston@uni-osnabrueck.de}

\thanks{
The first author is supported by Deutsche Forschungsgemeinschaft (DFG) -- Projektnr.\ 445466444}

\begin{document}

\maketitle

\begin{abstract}
We compute the reach, extremal curvature and volume of a tubular neighborhood for the Segre--Veronese variety intersected with the unit sphere.
\smallskip

\end{abstract}

\section{Introduction}
In this paper we study the metric geometry of rank-one tensors. More specifically, we compute the reach and the volume of a tubular neighborhood of the \emph{Segre--Veronese variety}, i.e.~the variety of rank-one tensors in the space of partially symmetric tensors. Since rank-one tensors form a cone, we intersect the Segre--Veronese variety with the unit sphere, thus obtaining the \emph{(spherical) Segre--Veronese manifold}; the proof that this is a manifold is provided below. We first describe the setting.

Let $H_{n,d}$ denote the vector space of homogeneous polynomials in $n+1$ variables $x_0,\ldots, x_n$ of degree $d$.
We consider the Bombieri-Weyl inner product $\langle\ ,\ \rangle$ on $H_{n,d}$: this is the inner product corresponding to the orthogonal basis vectors $\vect m_\alpha := \sqrt{\tbinom{d}{\alpha}}\,\vect x^{\alpha}$, where $\tbinom{d}{\alpha}  = \frac{d!}{\alpha_0!\cdots \alpha_n!}$ is the multinomial coefficient for $\alpha = (\alpha_0,\ldots,\alpha_n)$. The reason for this choice is that the Bombieri-Weyl inner product is invariant under an orthogonal change of coordinates; this was proved by Kostlan \cite{kostlan:93, Kostlan}.
The norm of a polynomial $\vect f\in H_{n,d}$ is  $\Vert \vect f\Vert  =\sqrt{\langle \vect f, \vect f\rangle}$, and the sphere is $\mathbb S(H_{n,d}):=\{\vect f\in H_{n,d} \mid\Vert \vect f\Vert =1\}$. 

For~$n,d\geq 1$, we denote the real variety of powers of linear forms in $H_{n,d}$ by
$$\widehat{\mathbb V}_{n,d} := \{\pm \, \bm\ell^d\mid \bm\ell \text{ is a linear form in } x_0,\dots,x_n\}.$$
The hat indicates that $\widehat{\mathbb V}_{n,d}$ is the cone over a spherical variety $$\Vnd:=\widehat{\mathbb V}_{n,d}\cap \mathbb S(H_{n,d}),$$ which we call the \emph{(spherical) Veronese variety}. Its dimension is $\dim \Vnd = n$.

We fix $r$-tuples of positive integers $\vect d=(d_1,\ldots,d_r)$ and $\vect n=(n_1,\ldots,n_r)$ and write 
$$\V := H_{n_1,d_1} \otimes \cdots \otimes  H_{n_r,d_r}.$$
The elements in $\V$ are called 
\emph{partially symmetric tensors}. They are multihomogeneous forms in $r$ sets of variables. The number $d=d_1+\cdots +d_r$ is called the \emph{total degree} of the tensors in~$\V$. 
For a tensor $\vect F = \sum_{\alpha_1,\ldots,\alpha_r}  F_{\alpha_{1},\ldots, \alpha_{r}}\; \vect m_{\alpha_{1}} \otimes \cdots\otimes\vect m_{\alpha_{r}}\in \V$, we use the short form $\vect F = (F_{\alpha_{1},\ldots, \alpha_{r}})$.
The following defines an inner product on $\V$:
$$\langle \vect F, \vect G \rangle := \sum_{\alpha_1,\ldots,\alpha_r} F_{\alpha_{1},\ldots, \alpha_{r}} \cdot G_{\alpha_{1},\ldots, \alpha_{r}}, \quad \text{where } \vect  F = (F_{\alpha_{1},\ldots, \alpha_{r}}), \vect G = (G_{\alpha_{1},\ldots, \alpha_{r}})\in\V.$$
With this, $\V$ becomes a Euclidean space, and we can measure volumes and distances in $\V$. The norm of a tensor $\vect F\in\V$ is $\Vert \vect F\Vert := \sqrt{\langle \vect F,\vect F\rangle}$, and the angular distance is~$d_{\mathbb S}(\vect F, \vect G):= \arccos \langle \vect F,\ \vect G\rangle$ for  $\vect F,\vect G$ in the unit sphere $ \mathbb S(\V)\subset \V$.

The \emph{(spherical) Segre--Veronese variety} in $\V$ is 
\begin{equation}\label{def_segre_veronese}
\Xnd := \left\{\vect f_1\otimes \cdots \otimes \vect f_r\mid \vect f_i\in \widehat{\mathbb V}_{n_i,d_i}\right\}\cap \mathbb S(\V).
\end{equation}
This is the variety of products of powers of linear forms in $\V$ that have unit norm. Tensors in~$\Xnd$ are also called \emph{decomposable}, \emph{simple} or \emph{rank-one} tensors. We prove in Proposition \ref{prop_helpful} that $\Xnd$ is an embedded smooth submanifold of $\mathbb S(\V)$ of dimension $\dim \Xnd=n_1+\cdots+n_r$; hence, we call $\Xnd$ a (spherical) Segre--Veronese \emph{manifold}.

The main focus of this paper is the {reach} and the volume of a {tubular neighborhood} of~$\Xnd$. We briefly recall what these are: the \emph{medial axis} $\mathrm{Med}(S)$ of a subset $S\subset \mathbb S(\V)$ is the set of all points $\vect F\in\mathbb S(\V)$ such that there exist at least two points $\vect G_1,\vect G_2 \in S$ with $d_{\mathbb S}(\vect F, S) = d_{\mathbb S}(\vect F, \vect G_i)$, $i=1,2$. The \emph{reach} of a subset $S$ is its minimum distance to the medial axis:
$$\tau(S):=\inf_{\vect F \in S}d_{\mathbb S}(\vect F, \mathrm{Med}(S)).$$
This notion was first introduced by Federer \cite{reach}. 

In our first main theorem we calculate the reach of the (spherical) Segre--Veronese manifold.
\begin{theorem}[The reach of (spherical) Segre--Veronese manifolds]\label{main:thm}
Let $\vect d=(d_1,\ldots,d_r)$ and $\vect n=(n_1,\ldots,n_r)$ be $r$-tuples of positive integers, and let $d:=d_1+\cdots+ d_r\geq 2$ be the total degree. The reach of the (spherical) Segre--Veronese manifold is 
$$\tau(\Xnd)= 
\begin{cases} 
\frac{\pi}{4}, & d\leq 5\\[0.4em]
\sqrt{\frac{d}{2(d-1)}}, & d>5.
\end{cases}$$
In particular, the reach only depends on the total degree $d$ and not on the dimensions of the Veronese varieties $\mathbb V_{n_i,d_i}$. 
\end{theorem}
This extends a theorem by Cazzaniga, Lerario and Rosana \cite{CLR2023}, who proved this formula for the Veronese variety, which is the special case $r=1$. Another special case worth mentioning is $d_1=\cdots=d_r=1$, which corresponds to the \emph{Segre manifold}.

Since $\Xnd$ is a smooth submanifold of the sphere, its reach is the minimum of the inverse of its maximal curvature and its smallest bottleneck. We also compute these. The next theorem explains which curves in $\Xnd$ have maximal and minimal curvature; this is proved in Section \ref{ext_curv}.

\bigskip
\begin{theorem}[Extremal curvature of curves in Segre--Veronese manifolds]\label{main_thm_curv}
Let the total degree of the (spherical) Segre--Veronese manifold $\Xnd$ be $d=d_1+\cdots+d_r\geq 2$.  Consider a geodesic 
$$\gamma(t) = \gamma_1(t)\otimes \cdots \otimes \gamma_r(t) \in \Xnd.$$
\begin{enumerate}\item 
The maximum curvature of $\gamma(t)$ is $\sqrt{\tfrac{2(d-1)}{d}}.$ It is attained by geodesics where the~$\gamma_i(t)$ are geodesics  in $\mathbb V_{n_i,d_i}$ of constant speed $\Vert \gamma_i'(t)\Vert = \sqrt{\tfrac{d_i}{d}}$. \\[0.2em]
\item 
The minimal curvature is $\sqrt{\tfrac{2(d_\ell-1)}{d_\ell}}$, where $d_\ell=\min\{d_1,\ldots,d_r\}$. 
It is attained by geodesics where $\gamma_\ell(t)$ is a geodesic 
parametrized by arc length in $\mathbb V_{n_\ell,d_\ell}$ and the other $\gamma_i(t)$ are constant.
\end{enumerate} 
\end{theorem}

Our third main result concerns the volume of the tubular neighborhood
$$U(\varepsilon) := \{\vect F\in\mathbb S(\V)\ \Bigm|\ d_{\mathbb S}(\vect F, \Xnd) < \varepsilon\}.$$
In Section \ref{sec_volume} we compute this volume in terms of complete matchings in a weighted graph. For a tuple $(v_1,\ldots,v_r)$ of nonnegative integers let $G=(V,E)$ be the complete graph on $v:=v_1+\cdots+v_r$ vertices. Recall that the tuple of degrees is $\vect d = (d_1,\ldots, d_r)$. We define weights on $E$ as follows: the vertices of $G$ are partitioned into $r$ groups $V= \mathcal I_1\sqcup\cdots\sqcup\mathcal I_r$ of cardinalities $\vert \mathcal I_k\vert = v_k$. The weight $w(e)$ of an edge $e$ between vertices in group $\mathcal I_k$ is $w(e)=\frac{d_k-1}{d_k}$, and the weight of an edge across groups is $1$. Given a perfect matching $C\subset E$ we define its weight to be $w(C):=\prod_{e\in C} w(e).$
This defines the function
\begin{equation}\label{def_D}
D_{\vect d}(v_1,\dots,v_r) :=  (-1)^\frac{v}{2} \sum_{C\subset E\; \text{\normalfont perfect matching}} w(C).
\end{equation}
We now have the following result.
\begin{theorem}[Volume of a tubular neighborhood]\label{main_thm_vol}
Let $n =  \dim \Xnd$, $N = \dim(\mathbb S(\V))$ and $c=N-n$. Define
$J_i(\varepsilon) = \int_{0}^{\varepsilon}  (\sin \phi)^{N-n+2i-1} \cdot (\cos\phi)^{n-2i}\; \mathrm d \phi$
and 
$$\theta_i = \frac{\Gamma(\frac{c}{2})}{2^i\,\Gamma(i+\frac{c}{2})} \;  \sum_{\substack{v_1,\ldots,v_r\in\mathbb N:\ v_i\leq n_i\\[0.2em] v_1+\cdots+v_r = 2i}} D_{\vect d}(v_1,\ldots,v_r).$$
Then for $\varepsilon < \tau(\Xnd)$, we have 
$$\mathrm{vol}(U(\varepsilon)) = \frac{\sqrt{(d_1\cdots d_r)^n}}{2^{r-1}} \cdot \mathrm{vol}(\mathbb S^{n_1}) \cdots \mathrm{vol}(\mathbb S^{n_r}) \cdot \mathrm{vol}(\mathbb S^{c-1})\cdot \sum_{0\leq 2i\leq n}  \theta_i\cdot J_i(\varepsilon).$$
\end{theorem}
The proof of this theorem is based on  computing the Weingarten map of $\mathbb X_{\vect n,\vect d}$, which we do in Theorem \ref{lm:secFundForm}. We show that the Weingarten map of $\mathbb X_{\vect n,\vect d}$ admits a block structure where the diagonal blocks are the Weingarten maps of the Veronese factors.

At the end of Section \ref{sec:matchings} we compute the coefficients $\theta_i$ from Theorem \ref{main_thm_vol} for the Segre manifold $\Xnd$ where $\mathbf n = \mathbf 1_r:= (1,1,\ldots,1)$ and $\mathbf d = \mathbf 1_r$. This Segre manifold is the image of $\mathbb S^1\times \cdots\times \mathbb S^1$ under the Segre embedding.

\subsection*{Acknowledgements}
We thank Antonio Lerario and Andrea Rosana for helping in finding several references related to differential geometry and for carefully explaining their paper \cite{CLR2023} to us. We also thank Jan Draisma for a discussion which led to~Theorem \ref{main_thm_vol} and an anonymous referee for helping us to improve the paper.

\subsection*{Organization of the paper}
In Section \ref{sec_diff_geo} we discuss the differential geometry and curvature of manifolds defined by tensor products of vectors. We then apply the results from Section \ref{sec_diff_geo} in Section \ref{sec:SFF_SV} to study the curvature of the (spherical) Segre--Veronese manifold~$\mathbb X_{\vect n,\vect d}$. In particular, we work out the Weingarten map of $\mathbb X_{\vect n,\vect d}$. In Section \ref{sec:reach} we compute the reach and prove Theorems \ref{main:thm} and \ref{main_thm_curv}. Finally, in Section \ref{sec_volume}, we compute the volume of the tubular neighborhood and prove Theorem \ref{main_thm_vol}.

\section{Tensor products of Riemannian manifolds}\label{sec_diff_geo}

The tensor space $\mathbb R^{m_1+1}\otimes \cdots\otimes \mathbb R^{m_r+1}$ is a Euclidean space for the inner product defined by $\langle \vect x_1\otimes \cdots\otimes \vect x_r ,\, \vect y_1\otimes \cdots\otimes \vect y_r\rangle = \langle \vect x_1,\vect y_1\rangle \cdots  \langle \vect x_r,\vect y_r\rangle$, where $\langle \vect x,\vect y\rangle = \vect x^T\vect y$. Write $N:=(m_1+1)\cdots (m_r+1)-1$; then $\mathbb S^N$ is the sphere in $\mathbb R^{m_1+1}\otimes \cdots\otimes \mathbb R^{m_r+1}$.
We consider for $1\leq i\leq r$ a smooth embedded submanifold  $\mathbb M_i$ of the sphere $\mathbb S^{m_i}\subset \mathbb R^{m_i+1}$. We define the \emph{tensor product} of these manifolds to be 
$$\mathbb M_1\otimes \cdots \otimes \mathbb M_r:= \left\{\vect x_1 \otimes \cdots \otimes \vect x_r \mid \vect x_1\in \mathbb M_1,\ldots, \vect x_r\in \mathbb M_r\right\}.$$
\begin{proposition}\label{prop_helpful}
For $1\leq i\leq r$ let $\mathbb M_i$ be a smooth Riemannian submanifold of~$\mathbb S^{m_i}$ of dimension~$n_i$, and denote 
$\mathbb M:= \mathbb M_1\otimes \cdots \otimes \mathbb M_r.$
Furthermore, denote the tensor product map by $\psi: \mathbb M_1\times \cdots \times \mathbb M_r\to \mathbb M, (\vect x_1,\ldots, \vect x_r)\mapsto \vect x_1 \otimes \cdots\otimes  \vect x_r$.
Then:
\begin{enumerate}
\item $\mathbb M$ is a Riemannian submanifold of $\mathbb S^N$ of dimension $n_1+\cdots+n_r$.
\item The tangent space of $\mathbb M$ at $\vect x=\vect x_1 \otimes \cdots \otimes \vect x_r$ is
$$T_{\vect x} \mathbb M = T_{\vect x_1}\mathbb M_1\otimes \vect x_2 \otimes \cdots \otimes \vect x_r + \cdots + \vect x_1\otimes \vect x_2 \otimes \cdots \otimes T_{\vect x_r}\mathbb M_r.$$
\item $\psi$ is a local isometry.
\end{enumerate}
\end{proposition}
\begin{proof}
For $1\leq i\leq r$ let $\mathcal A_i = ( U_{i,j},\varphi_{i,j})_j$ be an atlas for $\mathbb M_i$ such that $\vect u\in U_{i,j}$ implies that the antipodal point $-\vect u\not\in U_{i,j}$. Such an atlas exists since  $\vect 0\not\in \mathbb M_i$. Define the open sets $ U_{j_1,\ldots,j_r} := \psi(U_{1,j_1}\times \cdots \times U_{r,j_r})$; then $\psi|_{U_{1,j_1}\times \cdots \times U_{r,j_r}}$ is an isomorphism, so we have an atlas for $\mathbb M$  with charts $U_{j_1,\ldots,j_r}$
and maps $(\varphi_{1,j_1}\times \ldots\times \varphi_{r,j_r}) \circ (\psi|_{U_{j_1}\times \cdots \times U_{j_r}})^{-1}$. This also shows that we have $\dim \mathbb M= \dim (\mathbb M_1\times\cdots\times \mathbb M_r)=n_1+\cdots +n_r$. The Riemannian structure on the ambient space~$\mathbb S^N$ induces a Riemannian structure on $\mathbb M$.

For the second statement, we use that $T_{(\vect x_1,\ldots,\vect x_r)} (\mathbb M_1 \times \dots \times \mathbb M_r) = T_{\vect x_1} \mathbb M_1\times \cdots\times T_{\vect x_r} \mathbb M_r.$ For $1\leq i\leq r$ let $\vect v\in T_{\vect x_i}\mathbb M_i$.
By multilinearity, the derivative of $\psi$ at $(\vect x_1,\ldots,\vect x_r)$ maps
$$\mathrm D_{(\vect x_1,\ldots,\vect x_r)}\psi (0,\ldots,0,\vect v,0,\ldots, 0) = \vect x_1\otimes\cdots\otimes \vect x_{i-1}\otimes \vect v\otimes \vect x_{i+1}\otimes\cdots\otimes \vect x_r.$$
This proves the second statement, since $T_{\vect x}\mathbb M$ is the image of $\mathrm D_{(\vect x_1,\ldots,\vect x_r)}\psi$.

Finally, for $\vect v\in T_{\vect x_i}\mathbb M_i$ and $\vect w\in T_{\vect x_j}\mathbb M_j$ we have
$$\langle \vect x_1\otimes\cdots\otimes \vect v\otimes\cdots\otimes \vect x_r,\; \vect x_1\otimes\cdots \otimes \vect w\otimes\cdots\otimes \vect x_r\rangle = \begin{cases} \langle \vect v,\vect w\rangle, &i=j \\ \langle \vect v,\vect x_i\rangle\, \langle \vect w,\vect x_j\rangle, & i\neq j\end{cases} .$$
Since $\langle \vect v,\vect x_i\rangle = \langle \vect w,\vect x_j\rangle=0$, this shows that the inner product between the images of $(0,\ldots,0,\vect v,0,\ldots, 0)$ and $(0,\ldots,0,\vect w,0,\ldots, 0)$ under $\mathrm D_{(\vect x_1,\ldots,\vect x_d)}$ is $\langle \vect v,\vect w\rangle$, if $i=j$, and~$0$ otherwise. This shows that the derivative $\mathrm D_{(\vect x_1,\ldots,\vect x_r)} \psi$ preserves inner products on a basis of $T_{(\vect x_1,\ldots,\vect x_r)} (\mathbb M_1 \times \dots \times \mathbb M_r)$ and hence is an orthogonal map. This proves the third statement. 
\end{proof}

Using the notation of Proposition \ref{prop_helpful} we can now write
\begin{equation}\label{def_segre_veronese2}
\Xnd = \begin{cases}
     \mathbb V_{n_1,d_1}\otimes \cdots \otimes \mathbb V_{n_r,d_r}, &\text{if one is $d_i$ odd;} \\[0.3em]
    (\mathbb V_{n_1,d_1}\otimes \cdots \otimes \mathbb V_{n_r,d_r})\ \cup\ -(\mathbb V_{n_1,d_1}\otimes \cdots \otimes \mathbb V_{n_r,d_r}), &\text{if all $d_i$ are even.} 
\end{cases}
\end{equation}
Furthermore, Proposition \ref{prop_helpful} implies that $\Xnd$ is a smooth submanifold of the sphere of dimension $\dim \Xnd = n_1+\cdots+n_r$.
Therefore, we will henceforth call it the \emph{(spherical) Segre--Veronese manifold}.

\subsection{The second fundamental form of a tensor product manifold}\label{sec:sff_tensor}
Recall that the second fundamental form $\mathrm{II}_{\vect x}$ of a Riemannian submanifold $\mathbb M\subset \mathbb S^N$ at a point $\vect x\in \mathbb M$ is the trilinear form $$\mathrm{II}_{\vect x}: T_{\vect x} \mathbb M\times T_{\vect x} \mathbb M\times N_{\vect x} \mathbb M\to\mathbb R,\quad (\vect v, \vect w, \vect a)\mapsto \left\langle \frac{\partial \vect v(\vect u)}{\partial \vect w}\Bigm|_{\vect u = \vect x},\ \vect a \right\rangle,$$
where $\vect v(\vect u)$ is a (local) smooth tangent field of $\mathbb M$ with $\vect v(\vect x)= \vect v$. 
For a fixed $\vect a\in N_{\vect x}\mathbb M$ the \emph{Weingarten map} is the linear map
$$L_{\vect a}: T_{\vect x} \mathbb M\to T_{\vect x}  \mathbb M,$$
such that $\mathrm{II}_{\vect x} (\vect v,\vect w,\vect a) = \langle \vect v, L_{\vect a}(\vect w)\rangle.$ 

The next proposition provides the Weingarten map for a tensor product of manifolds.

\begin{proposition}\label{prop_helpful2}
Let $\mathbb M_1,\ldots, \mathbb M_r$ be as in Proposition \ref{prop_helpful} and $\mathbb M = \mathbb M_1\otimes\cdots \otimes \mathbb M_r$. Consider a point $\vect x = \vect x_1\otimes \cdots \otimes \vect x_r\in \mathbb M$ and a normal vector $\vect a\in N_{\vect x}\mathbb M$. A matrix representation of the Weingarten map of $\mathbb M$ at $\vect x$ in direction $\vect a$ relative to orthonormal coordinates is
$$L_{\vect a} = \begin{bmatrix}
    L_{1} & L_{1,2} & \cdots & L_{1,r}\\
    (L_{1,2})^T & L_{2} & \cdots & L_{1,r-1}\\
     & & \ddots & \\
     (L_{1,r})^T & (L_{1,r-1})^T & \cdots & L_{r}
    \end{bmatrix}
    ,$$
where the matrices $L_{i,j}$ and $L_i$ are defined as follows: let $\vect v_1^{(i)},\ldots,\vect v_{n_i}^{(i)}$ be an orthonormal basis for the tangent space $T_{\vect  x_{i}} \mathbb M_i$.
\begin{enumerate}
\item The off-diagonal blocks are
$$L_{i,j} := \begin{bmatrix} \langle \vect  x_{1} \otimes \cdots \otimes  \vect v_k^{(i)}\otimes \cdots\otimes\vect v_\ell^{(j)} \otimes  \cdots \otimes \vect  x_{r}, \vect a\rangle \end{bmatrix}_{1\leq k\leq n_i, 1\leq \ell\leq n_j} \in\mathbb R^{n_i\times n_j}.$$
\item Write $R_i := \vect x_1\otimes \cdots \otimes \vect x_{i-1}\otimes N_{\vect x_i} \mathbb M_i \otimes \vect  x_{i+1}\otimes \cdots \otimes  \vect x_{r},$
and let the orthogonal projection of $\vect a$ onto $R_i$ be $\vect x_1\otimes \cdots \otimes \vect x_{i-1}\otimes\vect a_i \otimes \vect  x_{i+1}\otimes \cdots \otimes  \vect x_{r}$. 
Then $\vect a_i\in N_{\vect x_i}\mathbb M_i$, and $L_i\in\mathbb R^{n_i\times n_i}$ is a matrix representation of the Weingarten map $L_{\vect a_i}$ of $\mathbb M_i$ at $\vect x_i$ in direction $\vect a_i$ with respect to the orthonormal basis $\vect v_1^{(i)},\ldots,\vect v_{n_i}^{(i)}$.
\end{enumerate}
\end{proposition}
\begin{proof}
By Proposition \ref{prop_helpful}, $\vect  x_{1} \otimes \cdots \otimes  \vect v_k^{(i)}\otimes\cdots \otimes \vect  x_{r}$ for $1\leq i\leq r$ and $1\leq k\leq n_i$ is an orthonormal basis of $T_{\vect x}\mathbb M$. 
Fix tangent vectors $\vect v = \vect  x_{1} \otimes \cdots \otimes  \vect v_k^{(i)}\otimes\cdots \otimes \vect  x_{r}$ 
and
$\vect w:=\vect  x_{1} \otimes \cdots \otimes  \vect v_\ell^{(j)}\otimes\cdots \otimes \vect  x_{r}$. Furthermore, let $\vect v_k^{(i)}(\bf u_i)$ be a local smooth tangent field of $\mathbb M_i$ with $\vect v_k^{(i)}(\vect x_i) = \vect v_k^{(i)}$. Then we obtain a local smooth tangent field of $\mathbb M$ with $\vect v(\vect x) = \vect v$ by setting
$$\vect v(\vect u_1\otimes\cdots\otimes \vect u_r) := \vect  u_{1} \otimes \cdots \otimes  \vect v_k^{(i)}(\vect u_i)\otimes\cdots \otimes \vect  u_{r}.$$
By multilinearity,
$$\frac{\partial \vect v(\vect u)}{\partial \vect w} =
\begin{cases}
\vect  x_{1} \otimes \cdots \otimes  \frac{\partial \vect v_k^{(i)}(\vect u_i)}{\partial \vect v_\ell^{(i)}} \otimes \cdots \otimes \vect  x_{r}, &\text{ if } i= j;\\[1em]
\vect  x_{1} \otimes \cdots \otimes  \vect v_k^{(i)}\otimes \cdots\otimes\vect v_\ell^{(j)} \otimes  \cdots \otimes \vect  x_{r}, &\text{ if } i\neq j.
\end{cases}
$$
This shows that the off-diagonal blocks of $L_{\vect a}$ are the matrices $L_{i,j}$.

For the diagonal blocks $(i=j)$ we observe that $\vect x_{1} \otimes \cdots \otimes  \tfrac{\partial \vect v_k^{(i)}(\vect u_i)}{\partial \vect v_\ell^{(i)}} \otimes \cdots \otimes \vect  x_{r}\in R_i$, so  
\begin{align*}
\langle \vect v, L_{\vect a}(\vect w)\rangle = \mathrm{II}_{\vect x}(\vect v,\vect w, \vect a) &= \left\langle \vect x_{1} \otimes \cdots \otimes  \tfrac{\partial \vect v_k^{(i)}(\vect u_i)}{\partial \vect v_\ell^{(i)}} \otimes \cdots \otimes \vect  x_{r}, \vect a\right\rangle \\
&= \left\langle  \tfrac{\partial \vect v_k^{(i)}(\vect u_i)}{\partial \vect v_\ell^{(i)}},\vect a_i\right\rangle= \langle \vect v_\ell^{(i)}, L_{\vect a_i}(\vect v_k^{(i)})\rangle.\end{align*}
This settles the case $i=j$.
\end{proof}

\section{Geodesics and the second fundamental form of Segre--Veronese manifolds}\label{sec:SFF_SV}
We now use the results from the previous section to compute the second fundamental form and the Weingarten map for a (spherical) Segre--Veronese manifold $\mathbb X_{\vect n,\vect d}.$ The first step towards this goal is considering the Veronese manifold ($r=1$).

\subsection{Veronese manifolds}
The Bombieri-Weyl inner product on the space of homogeneous polynomials $H_{n,d}$ has the property that 
\begin{equation}\label{RKH}\langle \vect f,\ \boldsymbol \ell^d\rangle = \vect f(\ell_0,\ldots,\ell_n),  \text{ where } \boldsymbol\ell(\vect x) = \ell_0x_0+\cdots + \ell_nx_n;\end{equation}
that is, taking the inner product of $\vect f\in H_{n,d}$ with $\boldsymbol \ell^d\in\mathbb V_{n,d}$ evaluates $\vect f$ at the coeffcient vector of $\ell$. One calls $(\vect x, \vect y)\mapsto \langle \vect x,\vect y\rangle^d$ a \emph{reproducing kernel} for $H_{n,d}$.

Recall that the scaled monomials $\vect m_\alpha = \sqrt{\tbinom{d}{\alpha}}\,\vect x^{\alpha}$ form an orthonormal basis for the space of polynomials $H_{n,d}$. We first prove a lemma on the structure of the tangent space of the Veronese manifold.
\begin{lemma}\label{lemma_TS_veronese}
Consider $\vect m_{(d,0,\ldots,0)} = x_0^d\in \Vnd$. Then an orthonormal basis for the tangent space  $T_{x_0^d} \Vnd$ is $\{\vect m_{(d-1,1,0,\ldots,0)},\ldots,\vect m_{(d-1,0,\ldots,0,1)}\} = \{\sqrt{d}\,x_0^{d-1}x_k\}_{k=1}^n$. 
\end{lemma}
\begin{proof}
It follows from \cite[Theorem 18]{CLR2023} that $T_{x_0^d} \Vnd$ is spanned by $\sqrt{d}\,x_0^{d-1}x_k$, $1\leq k\leq n$. The fact that these monomials are orthonormal follows directly from the definition of the Bombieri-Weyl inner product.
\end{proof}
We denote the two linear spaces from \cite[Equation (30)]{CLR2023}:
\begin{align}\label{def_W_P}
P&:=\operatorname{span} \{\vect m_{\alpha} \mid \alpha_0 < d-2\} \quad\text{and}\quad 
W:=\operatorname{span} \{\vect m_{\alpha} \mid \alpha_0 = d-2\}.
\end{align} 
The spaces $P$ and $W$ are orthogonal to each other. Lemma~\ref{lemma_TS_veronese} implies the following.
\begin{lemma}\label{lemma_NS_veronese}
$N_{x_0^d}\Vnd = P\oplus W$. 
\end{lemma}

The next theorem follows from Equations (28) and (29) in \cite{CLR2023}.
\begin{theorem}\label{weingarten_veronese}
Let $\vect f\in N_{x_0^d} \Vnd = P\oplus W$ and $L_{\vect f}$ be the Weingarten map of $\Vnd$ at $x_0^d$ and $\vect f$.
\begin{enumerate}
\item If $\vect f\in P$, then $L_{\vect f} = 0$.
\item If $\vect f \in W$, then $L_{\vect f}$ can be represented in orthonormal coordinates by the matrix
$$L_{\vect f}  = \sqrt{\frac{d-1}{d}}\; \begin{bmatrix}
\sqrt{2}\cdot f_{1,1} & f_{2,1} & \cdots & f_{n,1}\\[0.2em]
f_{2,1} & \sqrt{2}\cdot f_{2,2}& \cdots & f_{n,2}\\[0.2em]
\vdots & \vdots & \ddots & \vdots\\[0.2em]
f_{n,1} & f_{n,2} & \cdots & \sqrt{2}\cdot f_{n,n}
\end{bmatrix},$$
where 
$$\vect f = \sum_{1\leq i<j\leq n} f_{i,j}  \sqrt{d(d-1)}\, x_0^{d-2}x_ix_j +  \sum_{1\leq i\leq n} f_{i,i} \sqrt{\frac{d(d-1)}{2}}\, x_0^{d-2}x_i^2.$$
\end{enumerate}
\end{theorem}

Recall that a random symmetric $n\times n$ matrix $L=(\ell_{i,j})$ is $L\sim \mathrm{GOE}(n)$ if $\ell_{i,j}\sim N(0,\tfrac{1}{2})$ for $i\neq j$ and $\ell_{i,i}\sim N(0,1)$ and all entries are independent (except for the symmetry condition). The probability density of $L$ is $(2\pi)^{-\frac{n(n+1)}{4}} \exp(-\tfrac{1}{2}\mathrm{Trace}(L^TL)).$
\begin{corollary}
Let $\vect f\in N_{x_0^d} \Vnd$ and $L_{\vect f}$ be the Weingarten map of $\Vnd$ at $x_0^d$ and $\vect f$. If~$\vect f$ is Gaussian with respect to the Bombieri-Weyl metric then 
$$L_{\vect f}  \sim \sqrt{\frac{2(d-1)}{d}} \,\mathrm{GOE}(n).$$
\end{corollary}

\subsection{Segre--Veronese manifold}
We now turn to the Segre--Veronese manifold $\mathbb X_{\vect n,\vect d}.$ We first show that $\mathbb X_{\vect n,\vect d}$ is a homogeneous space. This allows us to compute geodesics and the second fundamental form at the distinguished point 
$$\vect E := x_0^{d_1}\otimes \cdots \otimes x_0^{d_r} \in \mathbb X_{\vect n,\vect d}.$$

The Bombieri-Weyl inner product on $\V$ has the property  
\begin{equation}\label{inner_prod_rankone}\langle \vect f_1\otimes \cdots \otimes \vect f_r,\, \vect g_1\otimes \cdots \otimes \vect g_r\rangle := \langle \vect f_1, \vect g_1\rangle \cdots \langle \vect f_r, \vect g_r\rangle.
\end{equation}
Moreover, it is invariant under an orthogonal change of variables; i.e., for orthogonal matrices $Q_1\in O(n_1+1),\ldots, Q_r\in O(n_r+1)$ we have 
\begin{equation}\label{orth_invariance}
\langle (\vect f_1\circ Q_1)\otimes \cdots \otimes (\vect  f_r\circ Q_r),\, (\vect g_1\circ Q_1)\otimes \cdots \otimes (\vect  g_r\circ Q_r)\rangle = \langle \vect f_1\otimes \cdots \otimes \vect f_r,\, \vect g_1\otimes \cdots \otimes \vect g_r\rangle.
\end{equation}
This invariance was studied by Kostlan in \cite{kostlan:93, Kostlan}. We define the action of the group
\begin{equation*}
    G:=\mathbb Z/2\mathbb Z\times O(n_1+1)\times \cdots \times O(n_r+1)
\end{equation*} on $\V$ as the linear extension of the action
$$(\sigma, Q_1,\ldots, Q_r).(\vect f_1 \otimes \cdots \otimes \vect f_r) := (-1)^\sigma \, (\vect f_1\circ Q_1)\otimes \cdots \otimes (\vect  f_r\circ Q_r).$$
By (\ref{orth_invariance}), this action is isometric. Furthermore, the additional $\mathbb Z/2\mathbb Z$ factor makes it act transitively on $\Xnd$. We summarize this in the following lemma.
\begin{lemma}\label{lem_hom_space}
The group $G$ acts isometrically and transitively on $\Xnd$.
\end{lemma}

For our formulation of the Weingarten map of $\mathbb X_{\vect n,\vect d}$ we first need to obtain a certain decomposition of the normal space. We define the spaces 
\begin{align*}
\mathcal W_i &:=  \operatorname{span}\{ \vect m_{\alpha_1} \otimes \cdots \otimes \vect m_{\alpha_r} \mid  (\alpha_i)_0 = d_i-2,\; (\alpha_k)_0 = d_k \text{ for } k\neq i\},\\
\mathcal G_{i,j} &:= \operatorname{span}\{ \vect m_{\alpha_1} \otimes \cdots \otimes \vect m_{\alpha_r} \mid  (\alpha_i)_0 = d_i-1,\; (\alpha_j)_0 = d_j -1,\;  (\alpha_k)_0 = d_k ,\text{ for } k\neq i,j\}
\end{align*}
and set
\begin{equation}\label{WGP}
\begin{aligned}
\mathcal W &:= \bigoplus_{1\leq i \leq r} \mathcal W_{i},\\
\mathcal G &:= \bigoplus_{1\leq i<j\leq r} \mathcal G_{i,j},\\ 
\mathcal P &:= (\mathcal W \oplus \mathcal G)^\perp \cap N_{\vect E} \mathbb X_{\vect n,\vect d}.
\end{aligned} 
\end{equation}
The next result extends Lemma \ref{lemma_NS_veronese} to the case $r\geq 2$.
\begin{lemma}\label{lem_TS}
If $r\geq 2$, the normal space has the orthogonal decomposition $$N_{\vect E} \mathbb X_{\vect n,\vect d} = \mathcal P \oplus \mathcal W \oplus \mathcal G.$$
\end{lemma}
\begin{proof}
By the inner product rule for simple tensors~(\ref{inner_prod_rankone}), and since the monomials $\vect m_{\alpha}$ are orthogonal, the decomposition $\mathcal P \oplus \mathcal W \oplus \mathcal G$ is an orthogonal decomposition. 
Therefore, we only have to show that $\mathcal W,\mathcal G\subset N_{\vect E} \mathbb X_{\vect n,\vect d}$. 

Recall from (\ref{def_segre_veronese2}) that  $\Xnd = \mathbb V_{n_1,d_1}\otimes \cdots \otimes \mathbb V_{n_r,d_r}$, if there is at least one odd $d_i$, and that, if all $d_i$ are even, we have $\Xnd=(\mathbb V_{n_1,d_1}\otimes \cdots \otimes \mathbb V_{n_r,d_r})\ \cup\ -(\mathbb V_{n_1,d_1}\otimes \cdots \otimes \mathbb V_{n_r,d_r})$.
It follows from Proposition \ref{prop_helpful} (2) that
$$
    T_{\vect E} \mathbb X_{\vect n,\vect d} =  T_{x_0^{d_1}}\mathbb V_{n_1,d_1}\otimes x_0^{d_2}\otimes  \cdots \otimes x_0^{d_r} + \cdots + x_0^{d_1}\otimes x_0^{d_2}\otimes \cdots \otimes T_{x_0^{d_r}}\mathbb V_{n_r,d_r}.
$$
Lemma \ref{lemma_TS_veronese} implies that 
$$T_{\vect E} \mathbb X_{\vect n,\vect d} = \bigoplus_{\ell=1}^r \, \operatorname{span}\{ \vect m_{\alpha_1} \otimes \cdots \otimes \vect m_{\alpha_r} \mid  (\alpha_\ell)_0 = d_\ell-1,\;  (\alpha_k)_0 = d_k ,\text{ for } k\neq \ell\}.$$

The space $\mathcal W_i$ is spanned by simple tensors $\vect m_{\alpha_1} \otimes \cdots \otimes \vect m_{\alpha_r}$ such that the $i$-th factor~$\vect m_{\alpha_i}$ is orthogonal to both $T_{x_0^{d_i}} \mathbb V_{n_i,d_i}$ and $x_0^{d_i}$. Using~(\ref{inner_prod_rankone}), this already shows that $\mathcal W_i\perp  T_{\vect E} \mathbb X_{\vect n,\vect d}$. Consequently, $\mathcal W\subset  N_{\vect E} \mathbb X_{\vect n,\vect d}$.

The space $\mathcal G_{i,j}$ is spanned by 
simple tensors $\vect m_{\alpha_1} \otimes \cdots \otimes \vect m_{\alpha_r}$ such that the $i$-th factor $\vect m_{\alpha_i}$ is orthogonal to $x_0^{d_i}$ and the $j$-th factor $\vect m_{\alpha_j}$ is orthogonal to $x_0^{d_j}$. Since $T_{\vect E} \mathbb X_{\vect n,\vect d}$ is spanned by simple tensors that have at most one factor different from $x_0^{d_k}$, the inner product rule (\ref{inner_prod_rankone}) implies that $\mathcal G_{i,j}\perp T_{\vect E} \mathbb X_{\vect n,\vect d}$ for all $i,j$; hence, $\mathcal G\subset N_{\vect E} \mathbb X_{\vect n,\vect d}$.
\end{proof}

\begin{example}\label{ex:segre}
Let us work out the decomposition from Lemma \ref{lem_TS} in the case of the Segre manifold. This is the case $\vect d = \vect 1 = (1,\ldots,1)$; i.e., $\Xnd = \mathbb S^{n_1}\otimes \cdots  \otimes\mathbb S^{n_r}$. Since $H_{n, 1}\cong \mathbb R^{n+1}$, we can view elements in $\V = H_{n_1, 1}\otimes \cdots\otimes H_{n_r,1}$ as $r$-dimensional tensors $\vect F=(F_{i_1,\ldots,i_r})$, where $0\leq i_j\leq n_j$ for $1\leq j\leq r$.
Figure \ref{tensor_fig} shows an order-three tensor illustrating the case $r=3$. We have for the Segre manifold 
$$T_{\vect E} \mathbb X_{\vect n,\vect 1} = \{(F_{i_1,\ldots,i_r})\mid \text{ there is exactly one $i_j$ greater than zero}\} .$$
Moreover, $\mathcal W = \emptyset$ and 
\begin{align*}
\mathcal G &= \{(F_{i_1,\ldots,i_r})\mid \text{ there are exactly two $i_j$s greater than zero}\},\\
\mathcal P &= \{(F_{i_1,\ldots,i_r})\mid \text{ there are at least three $i_j$s greater than zero}\}. 
\end{align*}
Figure \ref{tensor_fig} shows the case for $r=3$; here, the tangent space of $\mathbb X_{\vect n,\vect 1}$ is shown in red, $\mathcal G$ in green and $\mathcal P$ in blue.
\end{example}

\begin{figure}
\begin{center}
\begin{tikzpicture}[on grid, scale = 0.8]  
  \fill[yslant=-0.5, purple!30] (0,0) rectangle (1,2);
  \fill[yslant=-0.5, purple!30] (1,3) rectangle (3,2);
  \fill[yslant=0.5, xslant=-1,purple!30] (6,3) rectangle +(-2,-1);
  \fill[yslant=0.5, xslant=-1,purple!30] (4,2) rectangle +(-1,-2);
  \fill[yslant=0.5,purple!30] (3,-1) rectangle +(1,1);
  \fill[yslant=-0.5,teal!30] (1,0) rectangle (3,2);
  \fill[yslant=0.5,teal!30] (3,-3) rectangle +(1,2);
  \fill[yslant=0.5, xslant=-1,teal!30] (6,2) rectangle +(-2,-2);
  \fill[yslant=0.5,teal!30] (4,-1) rectangle +(2,1);
  \fill[yslant=0.5,blue!30] (4,-3) rectangle +(2,2);

  \draw[yslant=-0.5, thick] (0,0) grid (3,3);
  \draw[yslant=0.5, thick] (3,-3) grid (6,0);
  \draw[yslant=0.5,xslant=-1, thick] (3,0) grid (6,3);
\end{tikzpicture}
\end{center}
\caption{\label{tensor_fig}
Illustration of the tangent and normal spaces of the (spherical) Segre manifold $\mathbb S^2\otimes \mathbb S^2 \otimes \mathbb S^2$. at $\vect E = x_0\otimes x_0\otimes x_0$. We see a $3\times 3\times 3$ tensor $\vect F$ with~$3^3 = 27$ entries indexed by $(i,j,k)$ for $0\leq i,j,k\leq 2$. The white entry has index $(0,0,0)$ and corresponds to the tensor $\vect E$. The tangent and normal spaces of the (spherical) Segre manifold are orthogonal to $\vect E$ and hence spanned by the entries in red, green and blue. The tangent space $T_{\vect E} \mathbb X_{(3,3,3), (1,1, 1)}$ corresponds to the $2+2+2=6$ red entries, where exactly one index of $(i,j,k)$ is greater than 0. The green entries give the component $\mathcal G$ from Example \ref{ex:segre}, where exactly two indices are greater than 0.  The blue entries correspond to $\mathcal P$, where all indices are greater than 0.} 
\end{figure}
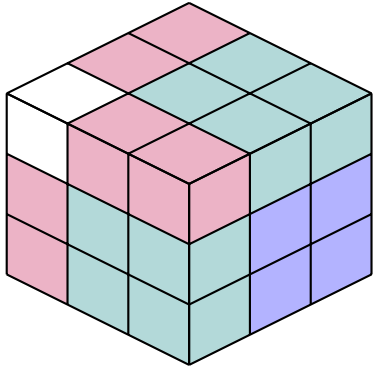

We can now prove a theorem on the structure of the Weingarten map of $\mathbb X_{\vect n,\vect d}$.
\begin{theorem}[Weingarten map of Segre-Veronese manifolds]\label{lm:secFundForm}
Consider a normal vector $\vect F \in N_{\vect E}\mathbb X_{\vect n,\vect d} = \mathcal P\oplus \mathcal W \oplus\mathcal G$ and let $L_{\vect F}$ be the Weingarten map of $\mathbb X_{\vect n,\vect d}$ at $\vect E$ and $\vect F$. Then~$L_{\vect F}$ 
is represented in orthonormal coordinates by the matrix
 \[
 L_{\vect F}= 
 \begin{bmatrix}
    L_{1} & L_{1,2} & \cdots & L_{1,r}\\
    (L_{1,2})^T & L_{2} & \cdots & L_{1,r-1}\\
     & & \ddots & \\
     (L_{1,r})^T & (L_{1,r-1})^T & \cdots & L_{r}
    \end{bmatrix}\in\mathbb R^{n\times n},\quad n=\dim \Xnd,
 \]
defined as follows: let us write~$\vect F = \vect P + \vect W + \vect G$, where $\vect P\in\mathcal P$, $\vect W\in\mathcal W$ and $\vect G \in\mathcal G$. Decompose further:
$$\vect W = \sum_{1\leq i\leq r} \vect W_i,\quad \vect G = \sum_{1\leq i<j\leq r} \vect G_{i,j},$$
where $\vect W_{i} = \vect m_{(d_1,0,\dots,0)} \otimes \cdots \otimes \vect m_{(d_{i-1},0,\dots,0)} \otimes \vect f_i\otimes \vect m_{(d_{i+1},0,\dots,0)} \cdots\otimes \vect m_{(d_{r},0,\dots,0)}\in\mathcal W_i$ with
\begin{align*}
    \vect f_i =& \sum_{1\leq k<\ell\leq n_i} f_{i, (k,\ell)} \,\vect m_{(d_i-2,0,\ldots, \underset{k\text{-th}}{1},\ldots,\underset{\ell\text{-th}}{1},\ldots,0)} + \sum_{1\leq k \leq n_i} f_{i, (k,k)} \,\vect m_{(d_i-2,0,\ldots,\underset{k\text{-th}}{2},\ldots,0)},
\end{align*}
and 
\begin{align*}
    \vect G_{i,j} = \sum_{k=1}^{n_i}\sum_{\ell = 1}^{n_j} g_{(i,j),(k,\ell)} \; \vect m_{(d_1,0,\ldots,0)}\otimes &\dots \otimes \vect m_{(d_i-1,0,\ldots, \underset{k\text{-th}}{1}, \ldots, 0)} \otimes \dots \\
    & \dots \otimes \vect m_{(d_j-1,0,\ldots, \underset{\ell\text{-th}}{1}, \ldots, 0)} \otimes \dots \otimes \vect m_{(d_r,0,\ldots, 0)} .
\end{align*}
Then
$$L_i = \sqrt{\frac{d_i-1}{d_i}} \begin{bmatrix}
\sqrt{2}\cdot f_{i,(1,1)} & f_{i,(1,2)} & \cdots & f_{i,(1,n_i)}\\[0.2em]
f_{i,(1,2)} & \sqrt{2}\cdot f_{i,(2,2)}& \cdots & f_{i,(2,n_i)}\\[0.2em]
\vdots & \vdots & \ddots & \vdots\\[0.2em]
f_{i,(1,n_i)} & f_{i,(2,n_i)} & \cdots & \sqrt{2}\cdot f_{i,(n_i,n_i)}
\end{bmatrix} \in\mathbb R^{n_i\times n_i}$$
and 
$$L_{i,j} = \begin{bmatrix}
g_{(i,j),(1,1)} & g_{(i,j),(1,2)} & \cdots & g_{(i,j),(1,n_j)}\\[0.2em]
g_{(i,j),(2,1)} &  g_{(i,j),(2,2)}& \cdots & g_{(i,j),(2,n_j)}\\[0.2em]
\vdots & \vdots & \ddots & \vdots\\[0.2em]
g_{(i,j),(n_i,1)} & g_{(i,j),(n_i,2)} & \cdots &  g_{(i,j),(n_i,n_j)}
\end{bmatrix} \in\mathbb R^{n_i\times n_j}.$$
(In particular, $L_{\vect F}$ depends only on the components $\mathcal W$ and $\mathcal G $. If $\vect F\in \mathcal P$, we have~$L_{\vect F} = 0$.)
\end{theorem}
\begin{proof}
Proposition \ref{prop_helpful2} implies the block structure of $L_{\vect F}$. The structure of the diagonal blocks is given by Theorem \ref{weingarten_veronese}. The structure of the off-diagonal blocks comes from the fact that $\vect m_{(d_i-1,0\cdots,\underset{k\text{-th}}{1},\cdots,0)}$ 
for $1\leq k\leq n_i$ are an orthonormal basis of the tangent space~$T_{x_0^{d_i}} \mathbb V_{n_i,d_i}$ by Lemma \ref{lemma_TS_veronese}. 
\end{proof}

An immediate corollary of Theorem \ref{lm:secFundForm} comes next.
\begin{corollary}\label{cor_LF}
Let $\vect F = (F_{\alpha_1,\ldots,\alpha_r})\in N_{\vect E}\mathbb X_{\vect n,\vect d}$ and $L_{\vect F}$ be the Weingarten map of $\mathbb X_{\vect n,\vect d}$ at~$\vect E$ in the normal direction $\vect F$. If $\vect F$ is Gaussian with respect to the Bombieri-Weyl norm, then
$$
 L_{\vect F}\sim 
 \begin{bmatrix}
    L_{1} & L_{1,2} & \cdots & L_{1,r}\\
    (L_{1,2})^T & L_{2} & \cdots & L_{1,r-1}\\
     & & \ddots & \\
     (L_{1,r})^T & (L_{1,r-1})^T & \cdots & L_{r}
    \end{bmatrix}
,$$
where 
$$L_k\sim \sqrt{\frac{2(d_k-1)}{d_k}} \, \mathrm{GOE}(n_k) \quad\text{and}\quad L_{i,j}\sim  \, N(0, I_{n_i}\otimes I_{n_j}),$$
and all blocks $L_k, L_{i,j}, 1\leq k\leq r, 1\leq i<j\leq r,$ are independent.
\end{corollary}

\begin{example}[Weingarten map of the Segre manifold]\label{example1} 
We consider again the case of the Segre manifold $\Xnd = \mathbb S^{n_1}\otimes \cdots \otimes \mathbb S^{n_r}$. In this case, the diagonal blocks of $ L_{\vect F}$ are all zero and the off-diagonal blocks are independent standard normal matrices. 
For instance, when $r=4$, $n_1 = n_2 = 2$ and $n_3 = n_4=1$ we have two $2\times 2$ diagonal zero blocks and two $1\times 1$ diagonal zero blocks:
\begin{align*}
 L_{\vect F} = \left[
\begin{array}{cc|cc|c|c}
    0 & 0 & \vect F_{1100} & \vect F_{1200} & \vect F_{1010} & \vect F_{1001}\\
    0 & 0 & \vect F_{2100} & \vect F_{2200} & \vect F_{2010} &\vect F_{2001} \\ \hline
    \vect F_{1100} & \vect F_{2100} & 0 & 0 & \vect F_{0110} & \vect F_{0101}\\
    \vect F_{1200} & \vect F_{2200} & 0 & 0 & \vect F_{0210} &\vect F_{0201} \\ \hline
    \vect F_{1010} & \vect F_{2010} & \vect F_{0110} & \vect F_{0210} & 0 & \vect F_{0011}\\\hline
    \vect F_{1001} & \vect F_{2001} & \vect F_{0101} & \vect F_{0201} & \vect F_{0011} & 0
\end{array} \right],
\end{align*}
where $\vect F = \sum_{i=0}^2\sum_{j=0}^2\sum_{k=0}^1\sum_{\ell=0}^1 \vect F_{ijk\ell} \; x_i\otimes x_j\otimes x_k\otimes x_\ell$. We will revisit this in Example~\ref{example_segre_L} below.
\end{example}

Finally, we provide an explicit description of geodesics in $\mathbb X_{\vect n,\vect d}$.
\begin{lemma}\label{lem_geodesics}
Let $(-\varepsilon,\varepsilon)\to\gamma(t)$ be a (locally defined) geodesic in $\mathbb X_{\vect n,\vect d}$ parametrized by arc length and passing through $\vect E$. Up to the action by the group $G$, the geodesic $\gamma(t)$ has the form
$$
\gamma(t) = \gamma_1(t) \otimes\cdots \otimes \gamma_r(t),
$$
where 
$$
\gamma_i(t)  = \big(\cos(d_i^{-1/2}\, a_i\,t)\, x_0 + \sin(d_i^{-1/2}\, a_i\,t)\, x_1\big)^{d_i}
$$
and $a_i,1\leq i\leq r$, are real numbers with $a_1^2+\cdots +a_r^2 = 1$.
\end{lemma}
\begin{proof}
We first show that $\gamma(t)$ as in the statement of the theorem is a geodesic.
The first derivative of $\gamma_i(t)$ at $t=0$ is 
\begin{equation}\label{dot_gammai}
\gamma'_i(0) = a_i \, \sqrt{d_i}\, x_0^{d_i-1}x_1 = a_i \, \vect m_{(d_i-1,1,0,\dots,0)}.
\end{equation}
The second derivative is \enlargethispage{\baselineskip}
\begin{equation}\label{dot_gammaii}
\begin{aligned}
\gamma''_i(0) &= -a_i^2 \, x_0^{d_i}  + a_i^2 \, (d_i-1) \,x_0^{d_i-2} x_1^{2}\\ 
&=  -a_i^2 \, \vect m_{(d_i,0,\dots,0)} + a_i^2 \,\sqrt{\tfrac{2(d_i-1)}{d_i}} \,\vect m_{(d_i-2,2,0,\dots,0)}.
\end{aligned}
\end{equation}
These formulas give the first and second derivatives of the factors of $\gamma(t)$.

Next, we compute the derivative of  $\gamma(t)$. The first derivative is 
$$\gamma'(t) = \sum_{i=1}^r \gamma_1(t) \otimes \dots \otimes \gamma_{i-1}(t) \otimes \gamma'_i(t) \otimes \gamma_{i+1}(t) \otimes \dots \otimes \gamma_r(t).$$
The second derivative is
\begin{align*}
    \gamma''(t) = &\sum_{i=1}^r \gamma_1(t) \otimes \dots \otimes \gamma_{i-1}(t) \otimes \gamma''_i(t) \otimes \gamma_{i+1}(t) \otimes \dots \otimes \gamma_r(t) \\
    &\quad + 2\sum_{1\leq i<j \leq r} \gamma_1(t) \otimes \dots  \otimes \gamma'_i(t) \otimes  \dots \otimes \gamma'_j(t) \otimes  \dots \otimes \gamma_r(t) .
\end{align*}
Recall from (\ref{WGP}) the definition of the spaces $\mathcal W$ and $\mathcal G$ and from Lemma \ref{lem_TS} that both spaces are contained in the normal space of $\Xnd$ at $\vect E$.
If we plug in the formulas (\ref{dot_gammai}) for the first derivative of $\gamma_i$ and (\ref{dot_gammaii}) for the second derivative, then we obtain the following expression:
\begin{equation}\label{ddot_gamma}
\gamma''(0) = -(a_1^2 + \cdots + a_r^2)\cdot \vect E + \vect W + \vect G,    
\end{equation}  
where $\vect W\in\mathcal W$ and $\vect G\in\mathcal G$  are given~by 
\begin{equation}\label{defWG}
\begin{aligned}
    \vect W &= \sum_{i=1}^r a_i^2 \sqrt{\tfrac{2(d_i-1)}{d_i}} \, \vect m_{(d_1,0,\ldots,0)} \otimes \dots \otimes \vect m_{(d_i-2,2,0,\dots,0)} \otimes \dots \otimes \vect m_{(d_r,0,\ldots,0)},\\
    \vect G &= 2\sum_{1\leq i<j \leq r} a_i a_j \, \vect m_{(d_1,0,\ldots,0)} \otimes \dots\otimes \vect m_{(d_i-1,1,0,\dots,0)} \otimes\\
    &\hspace{4cm}  \dots \otimes \vect m_{(d_j-1,1,0,\ldots,0)} \otimes \dots \otimes \vect m_{(d_r,0,\cdots,0)}.
    \end{aligned}
\end{equation}

For every $t_0\in(-\varepsilon,\varepsilon)$ there exists $U_0\in G$ such that $U_0\vect E  = U_0\gamma(0) = \gamma(t_0)$, and $N_{\gamma(t_0)}\Xnd = U_0 N_{\vect E}\Xnd$. Therefore, the second derivative of $\gamma(t)$ at $t=t_0$ satisfies
$$\gamma''(t_0) \in U_0\, (\mathbb R\cdot \vect E \oplus N_{\vect E}\Xnd) = \mathbb R\cdot \gamma(t_0) \oplus N_{\gamma(t_0)}\Xnd;$$
i.e., it has zero tangential component and hence is a geodesic. 

From Proposition \ref{prop_helpful} (3) it follows that 
$\Vert  \gamma'(t)\Vert^2 = \Vert \gamma_1'(t)\Vert^2 + \cdots + \Vert \gamma_r'(t)\Vert^2.$
Since $\Vert \gamma_i'(0)\Vert^2 = a_i^2$ by (\ref{dot_gammai}), the geodesic $\gamma(t)$ is parametrized by arc length if and only if 
$$a_1^2 + \cdots + a_r^2=1.$$
The statement of the theorem then follows from the local uniqueness of geodesics (see, e.g., \cite[Theorem 4.27]{Lee2018}) and the fact that we can find a suitable rotation in $G$ such that the tangent vector $\gamma_i'(0)$ points in the direction of $\sqrt{d_i} x_0^{d_i-1}x_1$. 
\end{proof}

\section{Reach of the Segre--Veronese manifold}\label{sec:reach}

We compute the reach $\tau(\Xnd)$ of the Segre--Veronese manifold. We adapt the strategy from \cite{CLR2023} and calculate the reach as the minimum of two quantities: 
$$\tau(\Xnd) = \min\{ \rho_1, \rho_2\},$$
where $\rho_1$ is the inverse of the maximal curvature of a curve in $\Xnd$ that is parametrized by arc length:
$$
\frac{1}{\rho_1} = \sup\left\{ \lVert P_{\vect E}(\gamma''(0)) \rVert \mid \gamma \text{ is a geodesic  in $\Xnd$ parametrized by arc length}\right\};
$$
here $P_{\vect E}$ denotes the orthogonal projection onto $T_{\vect E} \mathbb S(\V)$.
The second quantity, $\rho_2 ,$ is the width of the smallest bottleneck:
$$
\rho_2 =  \min \left\{\tfrac{1}{2}\, d_{\mathbb S}(\vect F,\vect E) \bigm| \vect F\in \Xnd,\ \vect F\neq \vect E\ \text{ and }\ \vect F-\vect E \in \mathbb R \cdot \vect E \oplus N_{\vect E} \Xnd\right\}.
$$
The goal of this section is to prove the following proposition, giving formulas for both~$\rho_1$ and $\rho_2$.

\begin{proposition}\label{rho:prop}
    Let $\vect d=(d_1,\ldots,d_r)$ and $\vect n=(n_1,\ldots,n_r)$ be $r$-tuples of positive integers, and let $d:=d_1+\cdots+ d_r\geq 2$. For the (spherical) Segre--Veronese manifold $\Xnd$ of total degree $d$, we have
    \begin{enumerate}
        \item $\rho_1 = \sqrt{\tfrac{d}{2(d-1)}}$.\label{rho1_eqn}\vspace{0.5em}
        \item $\rho_2 = \frac{\pi}{4}.$\label{rho2_eqn}
    \end{enumerate} 
\end{proposition}

We prove Proposition \ref{rho:prop} (\ref{rho1_eqn}) in Section \ref{ext_curv} and (\ref{rho2_eqn}) in Section \ref{bottlenecks}. Because the reach is the minimum of $\rho_1$ and $\rho_2$, this proves Theorem \ref{main:thm}.

\subsection{Extremal curvature of the Segre--Veronese manifold}\label{ext_curv} 

Let $\gamma(t)$ be a geodesic in $\Xnd$ parametrized by arc length. By orthogonal invariance (Lemma \ref{lem_hom_space}) we can assume that $\gamma(0)=\vect E$. 
As shown in Lemma~\ref{lem_geodesics}, geodesics in $\Xnd$ parametrized by arc length that pass through $\vect E$ can, without loss of generality, be written as 
$$\gamma(t) = \gamma_1(t) \otimes \dots \otimes \gamma_r(t) $$
where $
\gamma_i(t)  = (\cos(d_i^{-1/2}\, a_i\,t)\, x_0 + \sin(d_i^{-1/2}\, a_i\,t)\, x_1)^{d_i}
$
and $a_i,1\leq i\leq r$, are real numbers with  $a_1^2+\cdots +a_r^2 = 1$.
By (\ref{ddot_gamma}), we have $P_{\vect E}(\gamma''(0))  = \vect W + \vect G\in \mathcal W \oplus \mathcal G$, where the latter are defined in (\ref{defWG}). 
As $\mathcal W \perp \mathcal G$ and the $\vect m_{\alpha}$ form orthonormal bases, the magnitude of~$P_{\vect E}(\gamma''(0))$ is
\begin{align*}
    \lVert P_{\vect E}(\gamma''(0)) \rVert^2 &= \lVert \vect W \rVert^2 + \lVert \vect G \rVert^2 \\
    &= \sum_{i=1}^r a_i^4 \cdot \frac{2(d_i-1)}{d_i} + 4\sum_{1\leq i < j\leq r} a_i^2 a_j^2\\
    &= \sum_{i=1}^r a_i^4 \cdot \frac{2(d_i-1)}{d_i} + 2\sum_{i=1}^r a_i^2 \sum_{j \neq i}  a_j^2\\
    &= \sum_{i=1}^r \left( a_i^4 \cdot \frac{2(d_i-1)}{d_i} + 2 a_i^2(1-a_i^2) \right),\qquad \text{(because $\sum_{i=1}^r a_i^2=1$)}\\    
    &= 2\sum_{i=1}^r \left(a_i^2 - \frac{a_i^4}{d_i}\right) .
\end{align*}
To maximize this expression under the constraint $\sum_{i=1}^r a_i^2=1$ we consider the Lagrange function 
\begin{align*}
    \mathcal L(a_1,\dots,a_r,\lambda) := \sum_{i=1}^r \left(a_i^2 - \frac{a_i^4}{d_i} \right) - \lambda\left(1-\sum_{i=1}^r a_i^2\right) .
\end{align*}
Setting the derivatives of $\mathcal L$ to zero, we have
$$
    0 = \frac{\partial \mathcal L}{\partial a_i} = 2 a_i - \frac{4}{d_i} a_i^3 + 2\lambda a_i \quad 
    \Longrightarrow \quad a_i = \sqrt{\frac{d_i(1+\lambda)}{2}}  \quad\text{ or } \quad a_i = 0.
$$

Let us first consider the case when the $a_i$ are not equal to zero.  In this case, the equation $\sum_{i=1}^r a_i^2 = 1$ implies 
$$ 1 =  \sum_{i=1}^r a_i^2 = \sum_{i=1}^r \frac{d_i(1+\lambda)}{2} = \frac{d(1+\lambda)}{2} ,$$
 where $d=d_1+\cdots +d_r$ is the total degree. This shows $\lambda = \frac{2}{d} - 1$, so that 
 $$a_i = \sqrt{\frac{d_i}{d}}.$$ 
Thus, in this case
\begin{align*}
     \lVert P_{\vect E}(\gamma''(0)) \rVert &= \sqrt{ 2 \sum_{i=1}^r \frac{d_i}{d} - \frac{d_i}{d^2}} = \sqrt{\frac{2(d-1)}{d}}
\end{align*}
For the other critical values of $(a_1,\ldots,a_r)$ we get $\sqrt{\frac{2(d'-1)}{d'}}$, where $d' = \sum_{i\in I} d_i$ is the total degree of a subset $I\subset\{1,\ldots,r\}$ of factors. Since $x\mapsto \sqrt{\frac{2(x-1)}{x}}$ is an increasing function for $x\geq 1$, this shows that $\sqrt{\frac{2(d-1)}{d}}$ is indeed the maximal curvature. It also shows that $\sqrt{\frac{2(d_\ell-1)}{d_\ell}}$ is the minimal curvature, where $d_\ell = \min\{d_1,\ldots,d_r\}$. 
\begin{proof}[Proof of Theorem \ref{main_thm_curv}]
We have shown above that $\sqrt{\tfrac{2(d-1)}{d}}$ is the maximal curvature, and that $\sqrt{\tfrac{2(d_\ell-1)}{d_\ell}}$, where $d_\ell = \min\{d_1,\ldots,d_r\}$, is the minimal curvature. The geodesics that attain these curvatures are given by the critical values $a_i$, and, as shown by \cite{CLR2023}, $\gamma_i(t) = \big(\cos(d_i^{-1/2}\, a_i\,t)\, x_0 + \sin(d_i^{-1/2}\, a_i\,t)\, x_1\big)^{d_i}$ is a geodesic in $\mathbb{V}_{n_i,d_i}$.
\end{proof}

\subsection{Bottlenecks of the Segre--Veronese manifold}\label{bottlenecks}
We compute $\rho_2$, the width of the smallest bottleneck of the Segre--Veronese manifold.

Recall that $\rho_2$ is the minimum over the distances $\tfrac{1}{2}\, d_{\mathbb S}(\vect F,\vect E)$ where~$\vect F\in \Xnd$ with $\vect F\neq \vect E$ and $\vect F-\vect E\in \mathbb R\cdot \vect E\oplus N_{\vect E} \Xnd $. The latter is equivalent to 
$$\langle \vect F-\vect E, \vect G\rangle = 0\quad\text{for all}\quad \vect G\in T_{\vect E} \mathbb X_{\vect n,\vect d}.$$ We have 
$
    T_{\vect E} \mathbb X_{\vect n,\vect d} =  T_{x_0^{d_1}}\mathbb V_{n_1,d_1}\otimes x_0^{d_2}\otimes  \cdots \otimes x_0^{d_r} + \cdots + x_0^{d_1}\otimes x_0^{d_2}\otimes \cdots \otimes T_{x_0^{d_r}}\mathbb V_{n_r,d_r}
$ 
by Proposition \ref{prop_helpful}~(2).
We check that $\vect F-\vect E$ is orthogonal to each summand in this decomposition: 
let us write  
$$\vect F = \ell_1^{d_1} \otimes \cdots \otimes \ell_r^{d_r}$$
where the $\ell_i$ are linear forms and consider the inner product of $\vect F-\vect E$ with elements from the first summand in the decomposition of $T_{\vect E} \mathbb X_{\vect n,\vect d}$ above. By Lemma \ref{lemma_TS_veronese}
the monomials 
$x_0^{d_1-1}x_k$, for $1\leq k\leq n_1$, span the tangent space  $T_{x_0^{d_1}} \mathbb V_{n_1,d_1}$.
Consider $\vect G = (x_0^{d_1-1}x_k)\otimes x_0^{d_2}\otimes \cdots \otimes x_0^{d_r}\in T_{x_0^{d_1}} \mathbb V_{n_1,d_1}\otimes x_0^{d_2}\otimes \cdots \otimes x_0^{d_r}$. We have that 
\begin{align*}
\big\langle \vect F-\vect E,\ \vect G \big\rangle = \big\langle \vect F,\ (x_0^{d_1-1}x_k)\otimes x_0^{d_2}\otimes \cdots \otimes x_0^{d_r} \big\rangle & \stackrel{\text{by (\ref{inner_prod_rankone})}}{=}  \langle \ell_1^{d_1}, x_0^{d_1-1}x_k\rangle\, \prod_{i=2}^r \langle \ell_i^{d_i},x_0^{d_i}\rangle\\
&\stackrel{\text{by (\ref{RKH})}}{=} \langle \ell_1, x_0\rangle^{d_1-1}\, \langle \ell_1,x_k\rangle\, \prod_{i=2}^r \langle \ell_i,x_0\rangle^{d_i}.
\end{align*}
This inner product is zero for every $1\leq k\leq n_1$ if either $\ell_1=x_0$ or $\langle \ell_1,x_0\rangle =0$. 
We proceed similarly for the other summands in the decomposition of $T_{\vect E} \mathbb X_{\vect n,\vect d}$. 

Ultimately, we find that 
$\langle \vect F-\vect E, \vect G\rangle = 0$ for all $\vect G\in T_{\vect E} \mathbb X_{\vect n,\vect d}$ if and only if either $\ell_1=\cdots=\ell_r=x_0$ or there is at least one $\ell_i$ with $\langle \ell_i, x_0\rangle=0$, in which case $\langle \vect F,\vect E\rangle =0$ by (\ref{inner_prod_rankone}). Since $\vect F\neq \vect E$, it must be that the latter holds. \enlargethispage{\baselineskip}

Therefore, the bottlenecks of $\Xnd$ all have width $\arccos 0 = \frac{\pi}{2}$, so $\rho_2 = \frac{1}{2} \cdot \frac{\pi}{2} = \frac{\pi}{4}.$

\medskip
\section{Volume of the tubular neighborhood}\label{sec_volume}

Recall from Theorem \ref{main:thm} that the reach of the (spherical) Segre--Veronese manifold is~$\tau(\Xnd)=\frac{\pi}{4}$, if $d\leq5$, and $\tau(\Xnd)=\sqrt{\frac{d}{2(d-1)}}$, if $d>5$. 
In this section we prove Theorem \ref{main_thm_vol} by computing the volume of the tubular neighborhood for $\varepsilon < \tau(\Xnd)$
$$U(\varepsilon) := \left\{\vect F\in\mathbb S(\V)\ \Bigm|\ d_{\mathbb S}(\vect F, \Xnd) < \varepsilon\right\}.$$
The proof will be completed in Section \ref{sec:matchings} below.

For the computation we use \emph{Weyl's tube formula} \cite{Weyl39}.
We denote
$$n =  \dim \Xnd = n_1+\cdots + n_r, \quad N = \dim(\mathbb S(\V)),$$
and
$$J_i(\varepsilon) = \int_{0}^{\varepsilon}  (\sin \phi)^{N-n+2i-1} \cdot (\cos\phi)^{n-2i}\; \mathrm d \phi = \int_{t=0}^{\tan \varepsilon}  \frac{t^{N-n+2i-1}}{\left(1+t^2\right)^{\frac{N+1}{2}}}\; \mathrm d t.$$
Then Weyl's tube formula implies that the volume of $U(\varepsilon)$ is given as the following linear combination of the functions $J_i$:
\begin{equation}\label{weyl_tube_formula}\mathrm{vol}(U(\varepsilon)) = \sum_{0\leq 2i\leq n} \kappa_{i}\cdot J_{i}(\varepsilon),\end{equation}
with coefficients given by
$$\kappa_{i} = \int_{\vect G\in \Xnd} \left(\int_{\vect F \in N_{\vect G} \Xnd\ :\ \Vert \vect F\Vert = 1} m_{2i}(L_{\vect F}) \;\mathrm d \vect F\right)\; \mathrm d \vect G,$$
where $m_{2i}(L_{\vect F})$ denotes the sum of the $2i$-principal minors of the Weingarten map $L_{\vect F}$ in the normal direction $\vect F$. The coefficients $\kappa_{i}$ are called \emph{curvature coefficients}; they are isometric invariants of $\Xnd$.

It follows from Lemma \ref{lem_hom_space} that the integral in the formula for $\kappa_{i}$ is independent of~$\vect G$, so that 
\begin{equation}\label{formula_kappa}
\kappa_{i}= \mathrm{vol}(\Xnd) \; \int_{\vect F \in N_{\vect E} \Xnd\ :\ \Vert \vect F\Vert = 1} m_{2i}(L_{\vect F}) \;\mathrm d \vect F,
\end{equation}
where now the inner integral is over the sphere in the normal space of $\Xnd$ at the point $\vect E=x_0^{d_1}\otimes\cdots\otimes x_0^{d_r}$. The volume of the (spherical) Segre--Veronese manifold is computed next.
\begin{lemma}\label{lemma_volume_Xnd}
$\displaystyle \mathrm{vol}(\Xnd) = \frac{\sqrt{(d_1\cdots d_r)^n}}{2^{r-1}}\cdot \mathrm{vol}(\mathbb S^{n_1}) \cdots \mathrm{vol}(\mathbb S^{n_r})$.
\end{lemma}
\begin{proof}
We consider the map $\psi$ from Proposition \ref{prop_helpful} in the case of $\Xnd$.
If there is at least one odd $d_i$, the map $\psi$ is $2^{r-1}:1$. 
Proposition \ref{prop_helpful} (3) therefore implies  
$$\mathrm{vol}(\Xnd) \stackrel{\text{by (\ref{def_segre_veronese2})}}{=} \mathrm{vol}(\mathbb V_{n_1,d_1}\otimes \cdots \otimes \mathbb V_{n_r,d_r}) =  \frac{1}{2^{r-1}}\cdot  \mathrm{vol}(\mathbb V_{n_1,d_1}) \cdots \mathrm{vol}(\mathbb V_{n_r,d_r}).$$
On the other hand, if all $d_i$ are even, $\psi$ is $2^{r}:1$ and we have
$$\mathrm{vol}(\Xnd) \stackrel{\text{by (\ref{def_segre_veronese2})}}{=} 2\cdot \mathrm{vol}(\mathbb V_{n_1,d_1}\otimes \cdots \otimes \mathbb V_{n_r,d_r}) = 2\cdot \frac{1}{2^{r}}\cdot  \mathrm{vol}(\mathbb V_{n_1,d_1}) \cdots \mathrm{vol}(\mathbb V_{n_r,d_r}).$$
Finally, $\mathrm{vol}(\Vnd)  =\sqrt{d^n}\cdot \mathrm{vol}(\mathbb S^n)$ (see, e.g., \cite{EK1995}).
\end{proof}
\begin{remark}
The volume of the $k$-sphere is 
$$\mathrm{vol}(\mathbb S^k) = \frac{2 \pi^\frac{k+1}{2}}{\Gamma (\tfrac{k+1}{2})}.$$
\end{remark}

The main task in computing the volume of $U(\varepsilon)$ therefore is integrating the principal minors of the Weingarten map $L_{\vect F}$ over the normal space. For this we pass from the uniform distribution on the sphere to the Gaussian distribution. Since $L_{\lambda \cdot \vect F} = \lambda \cdot L_{\vect F}$ for $\vect F \in N_{\vect E} \Xnd$ and $\lambda \in\mathbb R$, we have 
\begin{equation}\label{formula_kappa_hom}
m_{2i}(L_{\lambda  \cdot  \vect F})=\lambda^{2i} \cdot m_{2i}(L_{\vect F}).
\end{equation}

Suppose that $\vect F$ is a Gaussian vector in the normal space, that is, a random tensor in $N_{\vect E}\Xnd$ with probability distribution $(2\pi)^{-\frac{c}{2}} \exp(-\tfrac{1}{2}\Vert \vect F\Vert^2)$. Then the two random variables $\Vert \vect F \Vert$ and $\vect F/\Vert \vect F \Vert$ are independent. We define the scalars
$$
\lambda_i :=  \mean_{\vect F \in N_{\vect E}\Xnd \text{ Gaussian }}\;\Vert \vect F\Vert^{2i}. 
$$
Using (\ref{formula_kappa_hom}) we can then pass between the uniform distribution on the sphere and the Gaussian distribution as follows:
$$\mean_{\vect F \in N_{\vect E}\Xnd \text{ Gaussian }}  m_{2i}(L_{\vect F})= \lambda_i\cdot \mean_{\vect F \in N_{\vect E}\Xnd \text{ uniform in the sphere }} m_{2i}(L_{\vect F}) .$$
Since $\Vert \vect F\Vert^2$ has a $\chi_c^2$-distribution with $c=\dim N_{\vect E} \Xnd$ degrees of freedom, $\lambda_i$ is the $i$-th moment of $\chi_c^2$; i.e., 
$$\lambda_i = 2^i\, \frac{\Gamma(i+\frac{c}{2})}{\Gamma(\frac{c}{2})}.$$
We have thus proved the following reformulation of (\ref{formula_kappa}).
\begin{lemma}\label{lemma_kappa2i}
Let $c=\dim N_{\vect E} \Xnd$. Then
$
\kappa_{i}= \mathrm{vol}(\Xnd) \cdot \mathrm{vol}(\mathbb S^{c-1}) \cdot   \, \theta_i,
$
where $$\theta_i := \frac{\Gamma(\frac{c}{2})}{2^i\,\Gamma(i+\frac{c}{2})}\cdot \mean_{\vect F \in N_{\vect E}\Xnd \; \mathrm{ Gaussian }} m_{2i}(L_{\vect F}).$$
\end{lemma}
For computing the expectation of $m_{2i}(L_{\vect F})$ we can rely on Corollary \ref{cor_LF}. Recall that this corollary implies that if $\vect F$ is Gaussian, $L_{\vect F}$ is a random symmetric matrix with independent blocks 
\begin{equation}\label{SFF_Xnd_formula}
    L_{\vect F} \sim 
\begin{bmatrix}
L_{1}  & \cdots & L_{1,r}\\
    &  \ddots & \\
    (L_{1,r})^T &  \cdots & L_{r}
\end{bmatrix},\quad 
\begin{array}{ll}
L_k\sim \sqrt{\tfrac{2(d_k-1)}{d_k}} \, \mathrm{GOE}(n_k),\\[1em]
L_{i,j}\sim  \, N(0, I_{n_i}\otimes I_{n_j}) .
\end{array}
\end{equation}
In general it is difficult to evaluate the expected value of the minors of this random matrix. We make an attempt using graph theory in the next subsection. 

\subsection{Perfect matchings in graphs and random determinants} \label{sec:matchings}
In this section we give a formula for $\mean m_{2i}(L_{\vect F})$ when $\vect F$ is Gaussian using concepts from graph theory. In the following, the degrees $\vect d=(d_1,\ldots,d_r)$ are fixed. Define the following random symmetric matrix with independent blocks:
$$L_{\vect d}(v_1,\ldots,v_r) := \begin{bmatrix}
L_{1}  & \cdots & L_{1,r}\\
    &  \ddots & \\
    (L_{1,r})^T &  \cdots & L_{r}
\end{bmatrix},\quad 
\begin{array}{ll}
L_k\sim \sqrt{\tfrac{2(d_k-1)}{d_k}} \, \mathrm{GOE}(v_k),\\[1em]
L_{i,j}\sim  \, N(0, I_{v_i}\otimes I_{v_j}) .
\end{array}
$$
This differs from (\ref{SFF_Xnd_formula}) in that we allow the sizes of the blocks to be arbitrary, not necessarily given by the dimension $n_i = \dim \mathbb V_{n_i,d_i}$. 
We can write the expected principal minors of $L_{\vect F}$ as 
$$\mean_{\vect F \in N_{\vect E}\Xnd \; \mathrm{ Gaussian }} m_{2i}(L_{\vect F}) = 
\sum_{\substack{v_1,\ldots,v_r\in\mathbb N:\ v_i\leq n_i\\[0.2em]v_1+\cdots+v_r = 2i}} \mean \det L_{\vect d}(v_1,\ldots,v_r).$$

Recall the definition of $D_{\vect d}(v_1,\ldots,v_r)$ from (\ref{def_D}): for a tuple $(v_1,\ldots,v_r)$ of nonnegative integers let $G=(V,E)$ be the complete graph on $v:=v_1+\cdots+v_r$ vertices where the vertices are partitioned into $r$ groups $V= \mathcal I_1\sqcup\cdots\sqcup\mathcal I_r$ of cardinalities $\vert \mathcal I_k\vert = v_k$. The weight $w(e)$ of an edge between vertices in group $\mathcal I_k$ is $w(e)=\frac{d_k-1}{d_k}$. The weight of an edge across groups is $1$. Given a perfect matching $C\subset E$ its weight is~$w(C):=\prod_{e\in C} w(e).$
Then, 
$$D_{\vect d}(v_1,\ldots,v_r) = (-1)^\frac{v}{2} \sum_{C\subset E\; \text{\normalfont perfect matching}} w(C) .$$
The main goal of this section is to prove the following characterization of the function~$D_{\vect d}$. In combination with (\ref{weyl_tube_formula}), Lemma \ref{lemma_volume_Xnd} and Lemma \ref{lemma_kappa2i} the next proposition completes the proof of Theorem \ref{main_thm_vol}.
\begin{proposition}\label{thm_graphs}
Let $(v_1,\ldots,v_r)$ be nonnegative integers. Then
$$D_{\vect d}(v_1,\dots,v_r) = \mean \det(L_{\vect d}(v_1,\ldots,v_r)).$$
\end{proposition}

\begin{example}\label{example_segre_L}
Recall from Example \ref{example1} that the random matrix for the Segre manifold $\Xnd = \mathbb S^2\otimes \mathbb S^2\otimes \mathbb S^1 \otimes \mathbb S^1$ (i.e., $n_1=n_2 = 2$ and $n_3=n_4=1$) is 
    \begin{align*}
 L_{\mathbf 1}(2,2,1,1) = \left[
    \begin{array}{cc|cc|c|c}
        0 & 0 & \vect F_{1100} & \vect F_{1200} & \vect F_{1010} & \vect F_{1001}\\
        0 & 0 & \vect F_{2100} & \vect F_{2200} & \vect F_{0201} &\vect F_{2001} \\ \hline
        \vect F_{1100} & \vect F_{2100} & 0 & 0 & \vect F_{0110} & \vect F_{0101}\\
        \vect F_{1200} & \vect F_{2200} & 0 & 0 & \vect F_{0210} &\vect F_{0201} \\ \hline
        \vect F_{1010} & \vect F_{2010} & \vect F_{0110} & \vect F_{0210} & 0 & \vect F_{0011}\\\hline
        \vect F_{1001} & \vect F_{2001} & \vect F_{0101} & \vect F_{0201} & \vect F_{0011} & 0
    \end{array} \right],
    \end{align*}
where the entries are all i.i.d.\ standard Gaussian. We compute the expected determinant of this matrix using Theorem \ref{thm_graphs}. The corresponding graph has $n_1+n_2+n_3+n_4=6$ vertices in four groups $\mathcal I_1=\{1,2\},\mathcal I_2=\{3,4\},\mathcal I_3=\{5\},\mathcal I_4=\{6\}$:
\begin{center}
    \begin{tikzpicture}
    \node[fill = purple!30] (1) at (-1,0) [circle,draw,inner sep=2pt] {1};
    \node[fill = teal!30] (3) at (3,0) [circle,draw,inner sep=2pt] {3};
    \node[fill = teal!30] (4) at (3,2) [circle,draw,inner sep=2pt] {4};
    \node[fill = blue!30] (6) at (-1,2) [circle,draw,inner sep=2pt] {6};
    \node[fill = purple!30] (2) at (1,-1) [circle,draw,inner sep=2pt] {2};
    \node[fill = orange!30] (5) at (1,3) [circle,draw,inner sep=2pt] {5};
    \draw (1) edge node {} (3);
    \draw (1) edge node {} (4);
    \draw (1) edge node {} (5);
    \draw (1) edge node {} (6);
    \draw (2) edge node {} (3);
    \draw (2) edge node {} (4);
    \draw (2) edge node {} (5);
    \draw (2) edge node {} (6);
    \draw (3) edge node {} (5);
    \draw (3) edge node {} (6);
    \draw (4) edge node {} (5);
    \draw (4) edge node {} (6);
    \draw (5) edge node {} (6);
    \end{tikzpicture}
    \end{center}
The edges within groups all have weight zero (they can be deleted). All other edges have weight one, so $D_{\mathbf 1}(2,2,1,1)=\mean\, \det\, L_{\mathbf 1} (2,2,1,1)$ is given by the negative of the number of perfect matchings in this graph. We can match $\{1,2\}$ with $\{3,4\}$. There are two possible such matches. Or we can match $1$ with either $5$ or $6$, in which case we have to match $2$ with either $3$ or $4$. There are 4 such matches. Finally, we can also match~$2$ with either $5$ or $6$, and by symmetry there are again 4 such matches. In total these are~10 matches, which shows that~$D_{\mathbf 1}(2,2,1,1)=-10$. 

\begin{figure}
\begin{center}
\includegraphics[width = 0.6\textwidth]{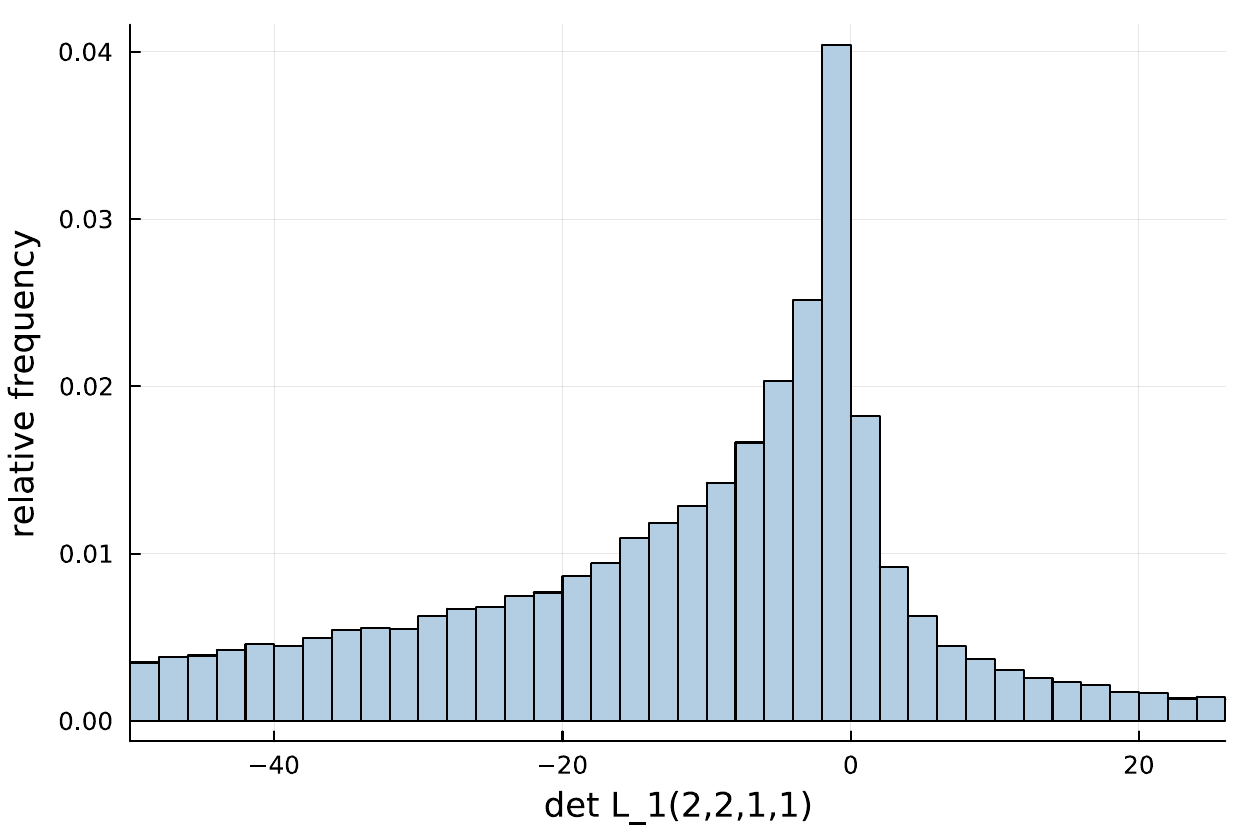}
\end{center}
\caption{The empirical distribution of $\det L_{\vect 1}(2,2,1,1)$ for $10^5$ sample points. The empirical mean of this sample is $-9.9995$. We show in Example \ref{example_segre_L} that the actual mean value is~$-10$.}
\end{figure}
\end{example}

\begin{proof}[Proof of Proposition \ref{thm_graphs}]
Let $v:=v_1+\cdots+v_r$ and write 
$$L_{\vect d}(v_1,\ldots,v_r) = (\ell_{i,j})_{1\leq i,j\leq v}.$$ Since the expectation is linear, Laplace expansion of the determinant yields 
$$\mean \;\det L_{\vect d}(v_1,\ldots,v_r) =  \sum_{\pi \in \mathfrak S_v} \mathrm{sgn}(\pi)\; \mean\prod_{i=1}^v \ell_{i,\pi(i)},$$
where $\mathfrak S_v$ is the symmetric group on $v$ elements. The $\ell_{i,\pi(i)}$ are all Gaussian with mean $\mean \ell_{i,\pi(i)} = 0$ and independent. This implies that the only terms whose expectation is not zero are those where all $\ell_{i,\pi(i)}$ appear as a square. In other words, only those expectations are not zero where $\pi\in \mathfrak S_v$ has the property that $\pi(i)\neq i$ for all $i$ and 
$\pi(i)=j$ implies $\pi(j)=i$. Such $\pi\in  \mathfrak S_v$ only exist when $v$ is even. 

If $v$ is odd, we therefore have $\mean \;\det L_{\vect d}(v_1,\ldots,v_r)=0$. Since for $v$ odd, no perfect matchings can exist, we also have $D_{\vect d}(v_1,\ldots,v_r)=0$.

If $v$ is even, on the other hand, the $\pi\in\mathfrak S_v$ with the above property are precisely products of~$\tfrac{v}{2}$ transpositions, so that
$$\mean \;\det L_{\vect d}(v_1,\ldots,v_r) = (-1)^\frac{v}{2}\sum_{\substack{\pi \in \mathfrak S_v:\\[0.2em]\pi \text{ is product of } \frac{v}{2} \text{ transpositions}}}  \mean\prod_{i=1}^v \ell_{i,\pi(i)}.$$
There is a $1:1$ correspondence between products of~$\tfrac{v}{2}$ transpositions and perfect matchings $C\subset E$, where $E$ is the set of edges in the complete graph $G=(V,E)$ on $v$ vertices. Let $C=\{(i_1,i_2), (i_3,i_4),\ldots, (i_{v-1},i_v)\}$ be the matching corresponding to $\pi$; i.e., for $j$ odd, $\pi(i_j)=i_{j+1}$. Then, using independence, we obtain
$$\mean\prod_{i=1}^v \ell_{i,\pi(i)} = \mean (\ell_{i_1,i_2}^2 \cdots \ell_{i_{v-1},i_v}^2) =  \mean \ell_{i_1,i_2}^2 \cdots \mean \ell_{i_{v-1},i_v}^2 = \sigma_{i_1,i_2}^2\cdots  \sigma_{i_{v-1},i_v}^2,$$
where $\sigma_{i_j,i_{j+1}}^2$ is the variance of $\ell_{i_j,i_{j+1}}$. By the definition of $L_{\vect d}(v_1,\ldots,v_r)$, the variance of the off-diagonal entries in the diagonal blocks is $\frac{d_k-1}{d_k}$, while the variance of the entries in the off-diagonal blocks  of $L_{\vect d}(v_1,\ldots,v_r)$ is $1$. That is:
$$\sigma_{i_j,i_{j+1}}^2 = \begin{cases}  \frac{d_k-1}{d_k}, & i_j,i_{j+1} \in\mathcal I_k\\1,& i_j,i_{j+1} \text{ are in different groups of vertices} \end{cases},$$
which shows that $\mean\prod_{i=1}^v \ell_{i,\pi(i)} = w(C)$, so  $D_{\vect d}(v_1,\ldots,v_r)=\mean \;\det L_{\vect d}(v_1,\ldots,v_r).$
\end{proof}

The last example of our paper is the computation of the curvature coefficients  $\kappa_i$ in Weyl's tube formula (\ref{weyl_tube_formula}) for the Segre manifold $\Xnd = \mathbb S^1\otimes \cdots \otimes \mathbb S^1\subset \mathbb S^{2^r -1}$.

\begin{example}\label{ex:segre2}
For the special case of the Segre manifold $\Xnd = \mathbb S^1\otimes \cdots \otimes \mathbb S^1$
we have $\vect d = \vect{1}_r = (1,\dots,1)$ and $ \vect n = \mathbf 1_r$.  In this case, Lemma \ref{lemma_volume_Xnd} yields
$$\mathrm{vol}(\Xnd) = \frac{1}{2^{r-1}}\cdot \mathrm{vol}(\mathbb S^{1}) \cdots \mathrm{vol}(\mathbb S^{1}) = 2\pi^r.$$
Furthermore, the codimension of $\Xnd$ is 
$$c=2^r - 1-r.$$ This implies that 
$\kappa_i = 2\pi^r \cdot \mathrm{vol}(\mathbb S^{c-1})\cdot \theta_i,$
where $\theta_i$ is defined as in Theorem \ref{main_thm_vol} and Lemma~\ref{lemma_kappa2i}. We compute the $\theta_i$. By Theorem~\ref{main_thm_vol}, we have
$\theta_i = \frac{\Gamma(\frac{c}{2})}{2^i\,\Gamma(i+\frac{c}{2})} \;  \sum D_{\vect d}(v_1,\ldots,v_r),$ where the sum is 
over all tuples  $(v_1,\ldots,v_r) \in \{0,1\}^r$ with
and $v_1+\cdots+v_r = 2i$. There are~$\binom{r}{2i}$ such tuples. Fix a tuple $(v_1,\ldots,v_r)$; this corresponds to the complete graph with $2i$ vertices where all edges have weight 1, and there are $(2i-1)!!$ perfect matchings on this graph. Therefore, by (\ref{def_D}) we have
$D_{\vect{1}_{2i}}(\vect{1}_{2i}) = (-1)^i (2i-1)!!$.
It follows that
$$\theta_i =(-1)^i\cdot\frac{\Gamma(\frac{c}{2})}{2^i\,\Gamma(i+\frac{c}{2})} \cdot \binom{r}{2i} \cdot  (2i-1)!!;$$
hence,
$$\kappa_i = (-1)^i \cdot 2\pi^r \cdot \mathrm{vol}(\mathbb S^{c-1}) \cdot\frac{\Gamma(\frac{c}{2})}{2^i\,\Gamma(i+\frac{c}{2})} \cdot \binom{r}{2i} \cdot  (2i-1)!! .$$
This completes the computation of the curvature coefficients of this Segre manifold.
\end{example}

\bibliographystyle{alpha}
\bibliography{math}

\end{document}